\documentclass[12pt]{amsart}
\usepackage{amsthm,amsmath,amssymb}

\makeatletter
\let\@wraptoccontribs\wraptoccontribs
\makeatother

\newcommand{\de}{\delta}
\newcommand{\ep}{\varepsilon}

\newcommand{\CC}{{\mathbb{C}}}

\newcommand{\HH}{{\mathbb{H}}}

\newcommand{\PP}{{\mathbb{P}}}
\newcommand{\QQ}{{\mathbb{Q}}}

\newcommand{\ZZ}{{\mathbb{Z}}}

\newcommand{\calO}{{\mathcal O}}

\newcommand{\calH}{{\mathcal H}}

\newcommand{\calA}{{\mathcal A}}

\newcommand{\calM}{{\mathcal M}}
\newcommand{\calJ}{{\mathcal J}}

\newcommand{\calX}{{\mathcal X}}

\newcommand{\calR}{{\mathcal R}}
\newcommand{\calRA}{{\mathcal RA}}

\newcommand{\op}{\operatorname}

\newcommand{\Sp}{\op{Sp}}

\newcommand{\Pic}{\op{Pic}}

\renewcommand{\t}{\theta}
\renewcommand\tt[2]{\t\left[\begin{matrix}#1\\ #2\end{matrix}\right]}
\newcommand{\oMg}{\overline{\calM_g}}
\newcommand{\oRg}{\overline{\calR_g}}
\newcommand{\oAg}{\overline{\calA_{g-1}}}

\theoremstyle{plain}
\newtheorem{thm}{Theorem}
\newtheorem{lm}[thm]{Lemma}
\newtheorem{prop}[thm]{Proposition}
\newtheorem{cor}[thm]{Corollary}

\theoremstyle{definition}

\newtheorem{rem}[thm]{Remark}

\begin{document}
\title{The Prym map on divisors, and the slope of $\calA_5$}
\author{Samuel Grushevsky}
\address{Mathematics Department, Stony Brook University,
Stony Brook, NY 11790-3651, USA.}
\email{sam@math.sunysb.edu}
\thanks{Research of  Samuel Grushevsky  is supported in part by National Science Foundation under the grant DMS-10-53313.}
\author{Riccardo Salvati Manni}
\address{Dipartimento di Matematica, Universit\`a ``La Sapienza'',
Piazzale A. Moro 2, Roma, I 00185, Italy}
\email{salvati@mat.uniroma1.it}

\contrib[with an appendix by]{Klaus Hulek}
\address{Institut f\"ur Algebraische Geometrie, Leibniz Universit\"at Hannover, Welfengarten 1, 30060 Hannover, Germany}
\email{hulek@math.uni-hannover.de}
\thanks{Research  of Klaus Hulek is supported in part by DFG grants Hu-337/6-1 and Hu-337/6-2}

\begin{abstract}
In this paper we compute the pullback of divisor classes under the Prym map (extended to the boundary), and apply this result to get a lower bound on the slope of effective divisors on
the perfect cone compactification of  the moduli space of principally polarized abelian fivefolds.

In the appendix
by Klaus Hulek, the notion of slope for arbitrary toroidal compactifications is discussed, and the slope bound is
shown to hold in general.
\end{abstract}
\maketitle

\section{Introduction}
The birational geometry of the moduli spaces of curves $\calM_g$ and of principally polarized
abelian varieties (ppav) $\calA_g$ has been studied extensively, with results for large enough $g$
starting with Harris and Mumford and Eisenbud and Harris' \cite{hamu,harris,eiha} proof that $\calM_g$ is of general type
for $g>23$ (followed by Farkas' proofs that the Kodaira dimension of $\calM_{23}$ is at least 2 \cite{fathesis} and that $\calM_{22}$ is of general type \cite{fa22}), and Tai's \cite{tai} and Mumford's \cite{mumforddimag} proofs that $\calA_g$ is of
general type for $g\ge 9$ and $g\ge 7$ respectively.

A more precise question is to describe the effective cone of $\calM_g$ ($\calM_g$ is of general type if the canonical class
is in its interior). The slope of the effective cone o $\calM_g$ for small genus was described by Harris and Morrison \cite{hamo},
who, led by their results, conjectured that the Brill-Noether divisor has minimal slope.
The minimal slope of $\calA_g$ (and its relation to $\calM_g$) for $g\le 4$ was described by the second author in \cite{sm}.
For the first non-classical cases, it turns out that the minimal slope of $\calA_4$ is given by the Schottky form ---
the divisor of the locus of Jacobians --- of slope 8 (see also \cite{hahu}), while recently the effective slope
of $\calM_g$ has attracted a lot of attention, after G.~Farkas and Popa \cite{fapo} disproved the slope
conjecture of Harris and Morrison \cite{hamo}. Despite the further work \cite{farkassyzygies,khosla},
the slope of the effective cone of $\oMg$ is not known for $g\ge 11$. It is known by the work of
Tai \cite{tai} (see also \cite{grAgsurvey}) that the slope of the effective cone of $\overline{\calA_g}$
approaches zero as $g\to\infty$. However, no examples of effective divisors of slope less than 6, on either $\calM_g$ or $\calA_g$, are known for any $g$, while it is not known if 6 is a lower bound for the slope of effective divisors on $\calM_g$.
(We also note that the nef cones have been studied for $\oMg$ --- see \cite{gikemo,fagi,gibney}
and for $\overline{\calA_g}$ --- see \cite{hulek,husaA4,shepherdbarron}). In the last few years G.~Farkas
and Verra, and also Ludwig, have studied the geometry of various covers and fibrations over the moduli
spaces of curves \cite{farkasevenspin,farkasverraoddspin,farkasverrainterm}.

In this paper we concentrate on the moduli space $\calR_g$ of Prym curves, the subject of a recent survey
\cite{faprymsurvey}: this is the moduli space of pairs
consisting of a smooth genus $g$ algebraic curve $X$ together with a line bundle $\eta$ on $X$ such that
$\eta^2=\calO_X$ while $\eta\ne\calO_X$. Such a data defines an \'etale double cover of $X$, and
the associated Prym variety is an element of $\calA_{g-1}$ --- thus we have a morphism $p:\calR_g
\to\calA_{g-1}$.

One then wants to extend this morphism to a suitable compactification. Indeed, if one takes a partial compactification $\calR_g^{part}$ of $\calR_g$ obtained by adding stable curves with one node, and partial compactification $\calA_{g-1}^{part}$ of $\calA_{g-1}$ obtained by adding semiabelic varieties of torus rank 1 (i.e.~for which the normalization is a $\PP^1$ bundle over an abelian variety, see \cite{mumforddimag}) the Prym map can naturally be extended to a morphism $\calR_g^{part}\to\calA_{g-1}^{part}$. Taking an actual compactification is trickier, as the structure depends on which toroidal compactification of $\calA_g$ is taken. Alexeev, Birkenhake, and Hulek \cite{albihu} studied the extension to a map to the second Voronoi toroidal compactification. Note, however, that no matter what toroidal compactification $\overline{\calA_{g-1}}$ of $\calA_{g-1}$, the morphism on the partial compactification (without defining it on $\overline{\calR_g}\setminus\calR_g^{part}$ gives a rational map $p:\overline{\calR_g}\dashrightarrow \overline{\calA_{g-1}}$.

In general the Picard group of an arbitrary toroidal compactification $\overline{\calA_g}$ may be very large, and is not known (see \cite{husaA4} for the discussion in genus 4). However, for the perfect cone compactification the Picard group (over $\QQ$) has rank two, and is generated by the class $L$ of the Hodge bundle, and the class $D$ of the (irreducible in this case) boundary divisor. We note also that the perfect cone compactification is $\QQ$-Cartier, as modular forms of weight $ k $ are a line bundle with divisor class $kL$, and cusp forms of weight $k$ are a line bundle of class $kL -D$, and thus any class $aL-bD$ is their linear combination.

{\em From now on we denote by $\overline{\calA_g}$ the perfect cone toroidal compactification of the moduli space of principally polarized abelian varieties, and consider the rational map $p:\overline{\calR_g}\dashrightarrow\overline{\calA_{g-1}}$.}

In this paper we use the Schottky-Jung relations to compute the pullback under $p$ of the theta-null divisor and thus compute the pullback map $p^*$
on divisors, and use this to bound the slope of the effective cone of
$\overline{\calA_5}$. Our main result is theorem \ref{thm:pull} 
computing the map $p^*$ 
on divisors, 
and the following bound for the slope:
\begin{thm}\label{thm:slope}
The minimal slope $s$ of effective divisors on $\overline{\calA_5}$  satisfies
$$
 7+\frac{5}{7}\ge s\ge 7+\frac{4198}{6269}\qquad ({\rm i.e.\ } 7.7142\ldots\ge s\ge 7.6696\ldots  ).
$$
\end{thm}
Recall that the slope of a divisor $aL-bD$ is defined to be $a/b$. The new statement in the theorem is the lower bound for the slope. The upper bound is provided by the Andreotti-Mayer divisor $N_0'$ (see remark \ref{rem:AM}).
Our result is thus the lower bound for the slope of effective divisors (or of Siegel modular forms) in genus 5, where no such bound was known  a priori.

\section*{Acknowledgements}
We would like to thank Gavril Farkas, Klaus Hulek and  Alessandro Verra  for very useful detailed discussions about the
geometry of $\oRg$ and of the geometry of the Prym map, and for their interest in this project. We are indebted to Klaus Hulek for reading a preliminary version of this manuscript very carefully, and for detailed comments and suggestions about the geometry of toroidal compactifications.

\section{The moduli of Prym curves and the Prym map}
The moduli space of curves $\calM_g$ admits a Deligne-Mumford compactification $\oMg$, the Picard
group of which is generated by the classes $\lambda_1$ (the determinant of the Hodge bundle),
$\delta_0$ --- the closure of the locus of irreducible nodal curves, and $\delta_i$ ---
the closure of the locus of nodal curves whose normalization has two components of genera $i$ and
$g-i$.

The moduli space $\calR_g$ of Prym curves admits a compactification $\oRg$ --- see \cite{faprymsurvey} for its history and details ---
which is a branched cover, which we denote $\pi:\oRg\to\oMg$. The
Picard group $\Pic_\QQ(\oRg)$ was described by Farkas and Ludwig in \cite{falu}, whose notation
and conventions we follow. It is generated  by the classes $\pi^*\lambda_1$ (which by abuse of notations
we denote $\lambda_1$), classes $\delta_0',\delta_0'',\delta_0^{ram}$, and various components of
the preimages of $\delta_i$ that will not be important to us.

>From the point of view of the moduli of ppav,
into which $\calM_g$ embeds by the Torelli map to $\calA_g$, $\calR_g$ embeds to $\calRA_g$. We recall the description of the boundary components of $\calRA_g$ and of $\calR_g$ from \cite[sec.~2.3]{donagiscju}. Indeed, consider the full level two cover $\calA_g(2)$, from which $\calRA_g$ can be obtained by taking a quotient by the stabilizer $G_\eta$ of a given two-torsion point $\eta$ within the symplectic group $\Sp(2g,\ZZ/2\ZZ)$:
$$
 G_\eta:=\lbrace \gamma\in\Sp(2g,\ZZ/2\ZZ)\mid \gamma\eta=\eta\mod \ZZ^{2g}\rbrace
$$
Then we have the covering maps $\calA_g(2)\to\calRA_g\to\calA_g$. The boundary components of $\calA_g(2)$ are indexed by non-zero two-torsion points $\mu$ on the ppav (or, correspondingly, on the Jacobian of the curve if one considers the boundary components of $\calM_g(2)$). The stabilizer $G_\eta$ acts on the set of two-torsion points $\mu\in (\ZZ/2\ZZ)^{2g}$ with three orbits, corresponding to whether $\mu\cdot\eta=0,\mu=\eta$, and $\mu\cdot\eta=1$, respectively (here we view the points $\mu$ and $\eta$ in $(\ZZ/2\ZZ)^{2g}$ endowed with the symplectic pairing). For each of these three orbits, the union of the boundary divisors of $\calA_g(2)$ corresponding to $\mu$ lying in this orbit is invariant under $G_\eta$ and thus descends to a union of boundary components of $\calRA_g$. It turns out that in fact each of these 3 is an irreducible boundary component of $\calRA_g$, which we thus denote $\delta_0',\delta_0'',\delta_0^{ram}$ correspondingly. By abuse of notation, following Donagi we also denote the same way the boundary components of $\calR_g$.

We refer the works of Donagi \cite{dofibers} and Donagi and Smith \cite{dosm} for the details on the structure of the Prym map for $g=6$, the case of most interest to us.

\smallskip
The map $\pi$ branches to order two along $\delta_0^{ram}$, and is unramified on any other divisor. Recalling that
the degree of $\pi$ is equal to $2^{2g}-1$, and counting the
number of such $\mu$ in each case, we can thus record the pullbacks and pushforwards under $\pi$ as follows
(this is of course well-known, see \cite{falu})
\begin{lm}\label{pi}
The pullback and pushforward under $\pi$ can be computed as follows:
$$\begin{aligned}
 \pi^*(\lambda_1)&=\lambda_1;&\pi^*(\delta_0)&=\delta_0'+\delta_0''+2\delta_0^{ram};\\
 \pi_*(\lambda_1)&=(2^{2g}-1)\lambda_1;&\pi_*(\delta_0')&=(2^{2g-1}-2)\delta_0;\\
 \pi_*(\delta_0'')&=\delta_0;&\pi_*(\delta_0^{ram})&=2^{2g-2}\delta_0.\\
\end{aligned}
$$
\end{lm}

\smallskip
We now recall that $\eta$ defines an \'etale double cover $\tilde X$ of a curve $X$ of genus $g$,
and the connected component of zero in the corresponding norm map of Jacobians $\op{Jac}(\tilde X)
\to\op{Jac}(X)$ naturally carries twice a principal polarization, and thus defines a Prym variety
$\op{Prym}(X,\eta)\in\calA_{g-1}$. This gives a morphism $p:\calR_g\to\calA_{g-1}$, the Prym map, which Friedman and Smith \cite{frsm} proved to be generically finite onto the image for $g\ge 6$, and the fibers of which in lower genus were the subject of a lot of research \cite{dosm,dofibers,izadiA4}. The Prym map can be
extended to a rational map $p:\oRg\dashrightarrow\oAg$.
The geometry, and in particular
the indeterminacy locus of such an extended map, for the case of the second Voronoi compactification, were studied in \cite{albihu}. Still, since $p$ defines a morphism from the partial compactification $\calR_g^{part}$ to the partial compactification $\calA_{g-1}^{part}$, for any toroidal compactification --- in particular for the perfect cone compactification --- the indeterminacy locus is of codimension at least two, and thus the pullback map $p^*$ on divisors is
well-defined. We recall (see \cite{tai,mumforddimag,hulek,husageometry} for more details) that the Picard group $\Pic_\QQ(\oAg)$ is generated by the class $L$ of the Hodge bundle and
the class $D$ of the boundary.

For future use, recall the structure of the Prym map $p$ along the boundary components of $\oRg$. First, for a generic curve $C=(C_1,p)\cup(C_2,q)/(p\sim q)$ in $\delta_i$, we note that the Jacobian does not depend on the points $p$ and $q$ (the components $\delta_i$ of $\partial\oMg$ are contracted in the Satake compactification). Thus for any $i>0$ the map $p$ contracts $\pi^{-1}(\delta_i)$ (more precisely, its open part, where it is defined) to a locus of codimension at least 2 in $\calA_{g-1}$, for any $g\ge 3$.

For the boundary components $\delta_0',\delta_0'',\delta_0^{ram}$ the situation is more delicate, and we recall it following \cite{beauville,donagiscju,dofibers,dosm,dosurvey,izadiA4}, see also \cite{faprymsurvey}. Let $(C,p,q)/(p\sim q)$ for $C\in\calM_{g-1}$ be a general point of $\delta_0$. Then for a general point of $\delta_0''$ lying over it the double cover is the Wirtinger double cover, and the corresponding Prym variety is simply the Jacobian of $C$, independent of the points $p$ and $q$. Thus $p$ maps (the open part of) $\delta_0''$ onto the locus of Jacobians $\calJ_{g-1}\subset\calA_{g-1}$. For genus $g\ge 6$ this locus is of codimension more than 1 in the Prym locus, i.e.~the map $p$ contracts $\delta_0''$. The double cover of a generic point of $\delta_0^{ram}$ is the so-called Beauville admissible double cover, i.e.~a double covers of $C$ branching at the points $p$ and $q$, and the corresponding Prym is an abelian variety. The branching order two here comes from the fact that the points $p$ and $q$ appear symmetrically in the construction.

The double cover corresponding to a generic point of $\delta_0'$ is inadmissible. This is to say that the corresponding double covering curve will have geometric genus $2g-3$, and two non-separating nodes. Thus the corresponding Prym variety will no longer be abelian --- rather it will be the semiabelic variety of torus rank one, obtained by compactifying a $\CC^*$ bundle over a Prym of $C$, where the line bundle corresponds to the difference $p-q$ under the Abel-Prym map. Thus for a fixed curve $C$ one gets a two-dimensional family of such semiabelic varieties corresponding to $S^2(\tilde C)/i$, where $i$ is the involution on the double cover $\tilde C$ of $C$. A dimension count then shows that for $g\ge 7$ the space of such Pryms is two-dimensional over the locus of Pryms of dimension $g-2$, and thus of total dimension $3g-1$, thus forming a divisor in the image $p(\oRg)$ (in fact an open part of this divisor is equal to the open part of the boundary $\partial p(\oRg)=p(\oRg)\cap\partial\oAg$). We note that $p(\delta_0')$ is in fact also a divisor in $\overline{\calA_5}$ (in fact equal to the boundary $D$) for $g=6$. Indeed, the map $p:\calR_5\to\calA_4$ has fibers generically of dimension two, and for a fixed point in $\calA_4$ there is a two-dimensional family of points in $\calR_5$ mapping to it. For each such curve there is a two-dimensional family of pairs of points on it, and thus for a generic point in $\calA_4$ we have a four-dimensional family in $\delta_0'\subset\overline{\calR_6}$ such that the abelian part of the corresponding semiabelic Prym is equal to this point in $\calA_4$ (and the degree of this map $\calR_{5,2}\to\calX_4$ to the universal family of ppav is equal to 27, see \cite{dofibers,izadiA4}).

As a result of this discussion, we get
\begin{prop}
For genus $g\ge 6$ the composition of the map $p^*$ on $\Pic_\QQ(\oAg)$ together with projecting to the span of $\delta_i$ and $\delta_0''$ is zero (i.e.~the classes $\delta_i$ and $\delta_0''$ do not appear in any pullbacks under $p^*$ of divisors on $\oAg$).
\end{prop}

\section{The pullback under the Prym map: the Schottky-Jung equation}
In this section we compute the pullback map $p^*:\Pic_\QQ(\oAg)\to\Pic_\QQ(\oRg)$. From now on we restrict ourselves to the case of $g\ge 6$, whence in view of the
above only the classes $\lambda_1,\delta_0'$, and $\delta_0^{ram}$ can appear in the pullbacks
$p^*L$ and $p^*D$ that we need to compute. Note that a generic point in $\delta_0^{ram}$ has a smooth Prym (we have $p:\delta_0^{ram}\dashrightarrow\calA_{g-1}$), while a generic point of $\delta_0'$ maps to a point in the boundary of $\oAg$. Thus we know a priori that
\begin{equation}\label{apriori}
 p^*L=a\lambda_1-b\delta_0^{ram};\qquad p^*D=c\delta_0'
\end{equation}
for some $a,b,c\in\QQ$, and the goal is to compute these coefficients. The standard method
would be to do some test curve computations, by choosing three curves in $\oRg$ and computing their
intersections with $\lambda_1,\delta_0',$ and $\delta_0^{ram}$, and the intersections of their images
under $p$ with $L$ and $D$. However, this seems rather tricky, as it requires knowing the explicit
geometry of the Prym map. Instead, we compute analytically the pullback of the theta-null divisor
by using the Schottky-Jung proportionality.

\smallskip
We recall that the Schottky-Jung proportionality, due classically to
Schottky \cite{schottky} and Schottky-Jung \cite{scju}, and in its modern form to H.~Farkas and
Rauch \cite{farascju}, relates the values of theta constants of the curve and of the Prym. We recall that
theta constants with characteristics $\ep,\de\in(\ZZ/2\ZZ)^g$, written as strings of zeroes
and ones, are defined as
$$
 \tt\ep\de(\tau):=\sum\limits_{n\in\ZZ^g}\exp\left(\pi i \left(n+\ep/2\right)^t
 \left(\tau\left(n+\ep/2\right)+\de\right)\right)
$$
where $\tau\in\calH_g$ is a point in the universal cover of $\calA_g$ --- the Siegel upper
half-space of symmetric $g\times g$ matrices with positive definite imaginary part.
Alternatively we can think of $m=\frac{\tau\ep+de}{2}\in (\ZZ/2\ZZ)^{2g}$ as a point of order two on the
ppav, and write this as $\theta_m(\tau)$. A characteristic $m$ is called even or odd depending on
whether the scalar product $\ep\cdot\de\in\ZZ/2\ZZ$ is 0 or 1, respectively. All
odd theta constants vanish identically, and there are $2^{g-1}(2^g+1)$ even characteristics.

For the Schottky-Jung relations, if $\eta$ is chosen to be
$\eta=\left[\begin{matrix}0&0&\ldots&0\\1&0&\ldots&0\end{matrix}\right]$, then the relation is
\begin{equation}
 \tt\ep\de(\sigma)^2= \op{const}\cdot\tt{0\,\ep}{0\,\de}(\tau)\cdot\tt{0\,\ep}{1\,\de}(\tau)
\end{equation}
where $\tau$ is the period matrix of the curve $C$, $\sigma$ is the period matrix of the Prym,
and the constant is independent of the characteristic $\ep,\de$.

In general for arbitrary
$\eta$ the Schottky-Jung relation can be obtained by applying modular transformations to
 the relation above. The result is as follows:
for a suitable embedding $j:(\ZZ/2\ZZ)^{2(g-1)}\hookrightarrow (\ZZ/2\ZZ)^{2g}$ such
that $\op{Im}(j)\subset \eta^\perp$, we have
\begin{equation}\label{SJ}
 \theta_n^2(\sigma)=\phi(n)\cdot\op{const}\cdot\theta_{j(n)}(\tau)\cdot\theta_{j(n)+\eta}(\tau),
\end{equation}
where $\phi$ is now an eighth root of unity depending on $n$ (this situation was
studied in more detail in \cite{farkashscju} for one particular choice of $\eta$, where $\phi(n)$ was derived explicitly).
We refer to \cite{faprymsurvey,grsurvey} for history, more details, and further references on the Schottky-Jung proportionalities.

\smallskip
We will now pull back the theta-null divisor divisor on $\oAg$ under the Prym map.
Recall that the theta-null divisor is defined as
\begin{equation}\label{tndef}
 \theta_{\rm null}:=\lbrace \tau\in \calA_g\mid \prod\limits_{m\in(\ZZ/2\ZZ)^{2g}_{\rm even}} \theta_m(\tau)=0\rbrace
\end{equation}
and its closure in $\overline{\calA_g}$ has class
\begin{equation}\label{tnclass}
 [\theta_{\rm null}]=2^{g-2}(2^g+1)L-2^{2g-5}D
\end{equation}
which follows from the fact that in genus $g$ there are $2^{g-1}(2^g+1)$ even theta constants, each of
which is a modular form of weight one half, together with a vanishing order computation
(see \cite{freitagbooksiegel,mumforddimag}). By using the Schottky-Jung proportionality (\ref{SJ}) we
compute the pullback of the sixteenth power of the theta-null divisor:
\begin{equation}\label{tnpull}
 8\pi^*[2\theta_{\rm null}(\sigma)]=8\left[\left\lbrace (\tau,\eta)\in\calR_g\mid \prod\limits_{m\in\eta^\perp_{\rm even}} \theta_m(\tau)=0\right\rbrace\right]
\end{equation}
(notice that the eighth root of unity in (\ref{SJ}) is torsion and does not matter for divisor class computations over $\QQ$).

The right-hand-side of the equation above is the zero locus of a product of $2\cdot 2^{g-2}(2^{g-1}+1)$ theta constants. Indeed,
for example for the standard choice of characteristic $\eta=\left[\begin{matrix}0&0&\ldots&0\\1&0&\ldots&0\end{matrix}\right]$, the set
$\eta^\perp_{\rm even}$ consists of all characteristics of the form $\left[\begin{matrix}0&\alpha\\x&\beta\end{matrix}\right]$
where $x$ is arbitrary,  and $\alpha,\beta$ is even. Since each theta constant is a modular form
of weight one half, we get for the pullback of twice the theta-null divisor, using (\ref{tnclass}) for $g-1$,
$$
 2p^*(2^{g-3}(2^{g-1}+1)L-2^{2g-7}D)=2^{g-2}(2^{g-1}+1)\lambda_1+\ldots
$$
It thus remains to compute the boundary coefficients of this pullback. Equivalently, this means computing the
slope of the modular form, so we need to compute the vanishing order of the right-hand-side of (\ref{tnpull}) near various
boundary components of $\oRg$. Recall (see \cite{vgeemenscju,donagiscju}) that boundary components themselves correspond
to points $\mu$ of order two, and the components $\delta_0'$ and $\delta_0^{ram}$ correspond to the cases of
$\mu\cdot\eta$ being 0 and 1, respectively (while $\mu\ne\eta$, which would be $\delta_0''$). For an analytic computation,
instead of fixing $\eta$ it is easier to fix $\mu$, i.e.~to choose a standard boundary component, and study the
degeneration there. This means we study the vanishing order of the right-hand-side of (\ref{tnpull}) as the period
matrix of the curve degenerates as $\tau\to \begin{pmatrix} i\infty& b^t\\ b&\tau'\end{pmatrix}$ for some $b\in\CC^{g-1}$
and some $\tau'\in\calH_{g-1}$. This corresponds to the case of the standard cusp
$\mu=\left[\begin{matrix}0&0&\ldots&0\\1&0&\ldots&0\end{matrix}\right]$.
A theta constant $\tt{x&\ep}{y&\de}(\tau)$ in such a limit has vanishing order 0 (approaches the generically non-zero
$\tt\ep\de(\tau')$) if and only if $x=0$, and has vanishing order $1/8$ if $x=1$ (in which case the lowest order term
in the Fourier-Jacobi expansion is $O(q^{1/8})$) along the boundary of $\oMg$. We refer to \cite{grhu2} for a related detailed discussion. Note, however, that since the map $\pi:\oRg\to\oMg$ has branching order two along $\delta_0^{ram}$, the local coordinate $w$ transverse to $\delta_0^{ram}$ on $\oRg$ is the square root of the local coordinate $q$ transverse to $\delta_0$ on $\oMg$, and thus the vanishing order on $\delta_0$ needs to be computed in terms of $q^{1/2}$ --- i.e.~will be double the vanishing order computed in terms of $q$.

Thus to compute the vanishing order of the right-hand-side of (\ref{tnpull}) at the standard cusp we simply need to
count the number of characteristics $m\in\eta^\perp_{\rm even}$ with $x=1$, multiplied by $1/8$. In the case of $\delta_0^{ram}$,
corresponding to $\mu\cdot\eta=1$, we can choose for example
$$
 \eta=\left[\begin{matrix}1&0&\ldots&0\\ 0&0&\ldots&0\end{matrix}\right],\qquad{\rm so\ that}\qquad
 \eta^\perp_{\rm even}=\left[\begin{matrix}x&\alpha\\0&\beta\end{matrix}\right]
$$
for arbitrary $x$ and any $\alpha,\beta\in(\ZZ/2\ZZ)^{2(g-1)}_{even}$. Thus we have $2^{g-2}(2^{g-1}+1)$ characteristics in
$\eta^\perp_{\rm even}$ with $x=1$, and the vanishing order here is $q^{2^{g-2}(2^{g-1}+1)/8}$. Thus along $\delta_0^{ram}$, with coordinate $w=q^{1/2}$ transverse to the boundary, the vanishing order is $2^{g-4}(2^{g-1}+1)$.

Similarly, for the case of $\delta_0'$, corresponding to $\mu\cdot\eta=0$, we can choose
$$
 \eta=\left[\begin{matrix}0&0&0&\ldots&0\\ 0&1&0&\ldots&0\end{matrix}\right],\qquad{\rm so\ that}\qquad
 \eta^\perp_{\rm even}=\left[\begin{matrix}x&0&\alpha\\y&z&\beta\end{matrix}\right]
$$
for $x,y,z\in\ZZ/2\ZZ$ and $\alpha,\beta\in(\ZZ/2\ZZ)^{g-2}_{\rm even}$ such that altogether the characteristic
is even, i.e.~$xy+\alpha\cdot\beta=0$. To compute the number of such characteristics with $x=1$, we note that for $x=1$
the parity is equal to $y+\alpha\cdot\beta$. Thus if one chooses $z,\alpha,\beta$ arbitrarily (so there are
$2\cdot 2^{2g-4}$ such choices), there exists a unique $y$ such that together with $x=1$ the overall characteristics
is even. Thus the total vanishing order here is equal to $1/8\cdot 2^{2g-3}$.

Using these vanishing orders, we finally get
\begin{prop}
For any $g\ge 6$ the pullback of the theta-null divisor under the Prym map is given by
$$\begin{aligned}
 p^*(2\theta_{\rm null})&= p^*(2^{g-2}(2^{g-1}+1)L-2^{2g-6}D)\\
&=2^{g-2}(2^{g-1}+1)\lambda_1-2^{g-4}(2^{g-1}+1)\delta_0^{ram}-2^{2g-6}\delta_0'.
\end{aligned}
$$
\end{prop}
For the sake of completeness, we also note that the vanishing order computation can also be done along $\delta_0''$, showing that the vanishing order is zero there. Thus even for $g\le 6$, when $p$ does not contract $\delta_0''$, this class would not appear in the pullback.

As a corollary, we get our first main result
\begin{thm}\label{thm:pull}
For any $g\ge 6$ the pullback of the divisor classes under the Prym map is given by
$$
 p^*L=\lambda_1-\frac14\delta_0^{ram};\quad p^*D=\delta_0'.
$$
\end{thm}
\begin{proof}
To prove the theorem we use the proposition above to determine the coefficients in (\ref{apriori}).
\end{proof}

\section{Bounds on the slope of $\calA_5$}
We recall that the slope of an effective divisor $E=aL-bD$ on $\oAg$ with $a,b\ge 0$ is defined to be $s:=a/b$. If $E$ is the zero divisor of a modular form $f$, then its slope is the quotient of its weight $a$ and its vanishing order $b$ at the boundary (so the slope is finite only if $f$ is a cusp form). Similarly, the slope of an effective divisor $E=a\lambda_1-b_0\delta_0-\sum b_i\delta_i$ on $\oMg$ is defined to be the minimum of all ratios $a/b_i$. Farkas and Popa \cite{fapo} proved that if the first ratio $a/b_0$ is sufficiently small, then the slope of $E$ is equal to $a/b_0$.

\begin{rem}\label{rem:AM}
An upper bound for the slope of the effective cone (the minimal slope of effective divisors) on $\overline{\calA_5}$ in theorem \ref{thm:slope} is provided by the Andreotti-Mayer divisor ---
the component of the locus of ppav with a singular theta divisor different from the theta-null \cite{anma}. We recall that its class was computed by Mumford \cite{mumforddimag} in arbitrary genus:
$$
 [N_0']=\left(\frac{(g+1)!}4+\frac{g!}{2}-2^{g-3}(2^g+1)\right)L-\left(\frac{(g+1)!}{24}-2^{2g-6}\right)D.
$$
Thus in genus 5 we have on $\overline{\calA_5}$ the following formula for its slope:
$$
 [N_0']=108L-14D;\qquad s(N_0')=\frac{108}{14}=7+\frac{5}{7}=7.7142\ldots
$$
yielding the upper bound in theorem \ref{thm:slope}.
\end{rem}

To obtain a lower bound for the slope of effective divisors on $\overline{\calA_5}$, we consider the map $\pi_*p^*$. Combining theorem \ref{thm:pull} and lemma \ref{pi} allows us to go from $\oAg$ to $\oMg$:
\begin{prop}
We have the following formulas for the map $\pi_*p^*$ on divisors:
$$
\begin{aligned}
 \pi_*p^*L&=\pi_*(\lambda_1-\delta_0^{ram}/4)=(2^{2g}-1)\lambda_1-2^{2g-4}\delta_0;\\
 \pi_*p^*D&=\pi_*(\delta_0')=(2^{2g-1}-2)\delta_0.
\end{aligned}
$$
\end{prop}
\begin{rem}
As a good consistency check for our formulas we can verify that the pullback of the theta-null divisor on $\oAg$ to $\oRg$, pushed down to $\oMg$, gives a multiple of the theta-null divisor. Indeed, we compute

$$\begin{aligned}
 \pi_*p^*(\theta_{\rm null})
 &=2^{g-3}(2^{g-1}+1)\left((2^{2g}-1)\lambda_1-(2^{2g-4}+2^{2g-4}-2^{g-3})\delta_0\right)\\
 &=2^{g-3}(2^{g-1}+1)\left((2^{2g}-1)\lambda_1-(2^{2g-3}-2^{g-3})\delta_0\right)\\
 &=2^{g-3}(2^{g-1}+1)(2^g-1)\left((2^{g}+1)\lambda_1-2^{g-3}\delta_0\right)\\
 &=(2^{g-1}+1)(2^g-1)\left(2^{g-3}(2^{g}+1)\lambda_1-2^{2g-6}\delta_0\right)
 \end{aligned}
$$
\end{rem}

The corollary of this computation that we are interested in is a lower bound for the slope of the effective cone of $\overline{\calA_5}$, proving the new inequality in theorem \ref{thm:slope}.
\begin{cor}
For any effective divisor $E=aL-bD$  on $\overline{\calA_5}$, with $a,b>0$, its slope $s=a/b$ satisfies
$$
 s\ge 7+\frac{4198}{6269},\qquad{\rm i.e.\ } s\ge 7.6696\ldots
$$
\end{cor}
\begin{proof}
Indeed, using the proposition we compute
$$
 \pi_*p^*[E]=\pi_*(p^*(aL-bD))=(2^{2g}-1)a\lambda_1-((2^{2g-1}-2)b+2^{2g-4}a)\delta_0.
$$
Since the map $p:\overline{\calR_6}\dashrightarrow\overline{\calA_5}$ is dominant, the class $\pi_*p^*[E]$
is the class of an effective divisor on $\overline{\calM_6}$ (notice that the class of the strict transform of $p^{-1}(E)\subset\overline{\calR_6}$ may of course not be equal to $p^*[E]$, but since we have $p^*[E]=[p^{-1}(E)]+{\rm Exceptional}$, we certainly know that $p^*[E]$ is the class of an effective divisor), and thus its slope is bounded below by
the minimal slope of effective divisors on $\overline{\calM_6}$. The minimal slope of effective divisors on $\overline{\calM_6}$ is equal to $\frac{47}6$, achieved by the Gieseker-Petri divisor (note that
$6+1$ is prime, so there is no corresponding Brill-Noether divisor). Thus we must have
$$
 \frac{(2^{2g}-1)a}{2^{2g-4}a+(2^{2g-1}-2)b}\ge \frac{47}{6}
$$
solving which for $g=6$ yields
$$
 s\ge \frac{48081}{6269}=7+\frac{4198}{6269}=7.6696\ldots
$$
\end{proof}

\begin{rem}
We also note a peculiar consequence of our computations for the case of $g=7$. Recall that the Kodaira dimension of $\overline{\calA_6}$ is not known. The slope of the canonical bundle $s(K_{\overline{\calA_6}})$ is equal to 7, and thus for the Kodaira dimension of $\overline{\calA_6}$ to be nonnegative is equivalent to the divisor $7L-D$ being effective. We then compute in this case for the slope
$$
 s(\pi_*p^*(7L-D))=s(7\cdot(2^{14}-1)\lambda_1-(7\cdot2^{10}+2^13-2)\delta_0)=7+\frac{1025}{2194}=7.4\ldots.
$$
However, the minimal slope of effective divisors on $\overline{\calM_7}$ is achieved by the Brill-Noether divisor, which has slope $6+\frac{12}{6+1}=7.7\ldots$. Thus any effective divisor on $\overline{\calA_6}$ of slope at most 7 must contain the Prym locus $p(\calR_7)$.
\end{rem}

\section*{Appendix, by Klaus Hulek}

In this appendix we will discuss the notion of {\em slope} for arbitrary toroidal compactifications
$\mathcal A_g^{\operatorname{tor}}$ of $\mathcal A_g$. Besides the perfect cone or first Voronoi compactification
$\overline{{\mathcal A}_{g}}=\mathcal A_g^{\operatorname{Perf}}$ discussed in the main paper, two other
toroidal compactifications have been considered by numerous authors,
namely the second Voronoi compactification $\mathcal A_g^{\operatorname{Vor}}$ and the central cone
compactification $\mathcal A_g^{\operatorname{cent}}$ which coincides with the Igusa compactification.

The perfect cone compactification has the special property that the boundary is irreducible and thus, over $\QQ$,
the group of Weil divisors coincides with the group of Cartier divisors, being generated by the Hodge line bundle
$L$ and the boundary $D$. For general toroidal compactifications the situation is much more difficult: the boundary is not necessarily
irreducible (this already happens for the second Voronoi compactification in genus $4$, where the
Picard number is $3$, see \cite{husaA4}), and thus one cannot expect that the variety $\calA_g^{\operatorname{tor}}$ is
necessarily $\QQ$-factorial. Nevertheless it is true that for the group
of Weil divisors one has
$$
\operatorname{Div} (\mathcal A_g^{\operatorname{tor}}) \otimes \QQ = \QQ L + \QQ D_0 +\sum_{i=1}^r D_i,
$$
where $D_0$ is the closure of the boundary of Mumford's partial compactification, which is contained in all
toroidal compactifications, and the $D_i, i \geq 1$ are further boundary divisors. Recall that toroidal compactifications
are determined by the choice of a suitable fan in the rational closure of the space of positive definite real
$(g \times g)$-matrices and that the $D_i$ are in 1-to-1 correspondence with $\Sp(2g,\ZZ)$-orbits of rays in this fan.
Here $D_0$ cor\-res\-ponds to the orbit of rank $1$ matrices. Any irreducible effective Weil divisor $H$, which is not a
boundary component, can be
written in the form
$H=aL - b_0D_0 - \sum_{i=1}^r b_iD_i$ with $a, b_i \geq 0$. In analogy with the case of $\overline {\mathcal M}_g$
we suggest to define the {\em slope} of $H$ by
$$
\operatorname{slope}(H)= \frac{a}{\min \{b_i\}}
$$
and the {\em slope of $\mathcal A_g^{\operatorname{tor}}$}  as the infinum of $\operatorname{slope}(H)$ for all
effective $H$.

The main observation of this appendix is
\begin{thm}\label{theo:mainappendix}
For every toroidal compactification we have
$$
\operatorname{slope}(\mathcal A_g^{\operatorname{tor}})=\operatorname{slope}(\mathcal A_g^{\operatorname{Perf}}).
$$
\end{thm}
\begin{proof}
Recall that toroidal compactifications are normal and thus regular in codimension $2$. It is thus enough to consider the
stack-smooth part of $\mathcal A_g^{\operatorname{tor}}$ where we have at most finite quotient singularities and where
we can also disregard the difference between Cartier and Weil $\QQ$-divisors. Any irreducible divisor $H$ which is not a
boundary divisor is given as the zero-locus of a modular form $F$. We thus have to understand the vanishing order of the
form $F$ on the different boundary divisors. In fact, we claim that the minimal vanishing order occurs on the
boundary component $D_0$ of Mumford's partial compactification. Such an argument was first used by Tai \cite[p.426 ff]{tai}
in his proof that ${\mathcal A}_g$ is of general type for $g\geq 9$.

In order to prove this
we first remark that the map from the partial compactification of the quotient of Siegel space
$\HH_g$ by the center of the unipotent radical of the normalizer of a cusp to $\mathcal A_g^{\operatorname{tor}}$
is unramified \cite[Section 5]{tai} at a general point of each boundary component.
Hence it is sufficient to consider the Fourier-Jacobi expansion of a modular form $F(Z)$ with respect to the cusp associated
to a $g' < g$ dimensional isotropic subspace of $\QQ^{2g}$ (this cusp is unique modulo the action of
$\Sp(2g,\ZZ)$).
This expansion is of the form
$$
f(Z)=\sum_S {\theta}_S(\tau,w) e^{2 \pi i \operatorname{tr}(Sz)}
$$
where $g'+g''=g$, $\tau \in \HH_{g'}$, $w \in \operatorname{Mat}(g' \times g'',\CC)$ and
$z \in \operatorname{Mat}(g'' \times g'',\CC)$. Here $S$ runs through all $(g'' \times g'')$-dimensional
semi-integral, semi-positive matrices. The claim then follows from a result of Barnes-Cohn \cite{BC}, which says that the
minimum value of $\operatorname{tr}(SX)$, where $X$ is an integral, positive or semi-positive matrix, is attained
at some matrix of rank one, i.e.~at $X=x^tx$ for some $0 \neq x \in \ZZ^{g''}$. In other words the vanishing order
is minimal at the boundary divisor $D_0$.
\end{proof}

\begin{cor}
Theorem \ref{thm:slope} holds for all toroidal compactifications of ${\mathcal A}_5$.
\end{cor}

\begin{rem}
We would like to point out that the proof of Theorem \ref{theo:mainappendix} does not generalize to
moduli spaces of non-principally polarized abelian varieties as in this case the integral structure of
the rational closure of the space of positive definite $(g \times g)$-matrices is not the standard one.
\end{rem}

\end{document}